\documentclass[11pt,reqno]{amsart}

\numberwithin{equation}{section}
\usepackage[all,cmtip]{xy}
\usepackage{amssymb,rotating,amsmath,amsfonts,amsthm}
\usepackage{color,graphics,epsfig,bbm,bm,manfnt,mathrsfs}

\newcommand{\calH}{\mathcal{H}}

\newcommand{\mC}{\mathbb{C}}
\newcommand{\mD}{\mathbb{D}}

\newcommand{\mN}{\mathbb{N}}

\newcommand{\mT}{\mathbb{T}}

\newcommand{\sH}{{{\mathscr{H}}^{\infty}}}

\newtheorem{theorem}{Theorem}[section]

\newtheorem{corollary}[theorem]{Corollary}
\newtheorem{proposition}[theorem]{Proposition}

\theoremstyle{definition}

\theoremstyle{definition}
\newtheorem{definition}[theorem]{Definition}

\theoremstyle{definition}

\begin{document}



\title[On the algebra of bounded Dirichlet series]{Some curiosities of \\
the algebra of bounded Dirichlet series}

\author{Raymond Mortini}
\address{Universit\'{e} de Lorraine\\ 
D\'{e}partement de Math\'{e}matiques et Institut \'Elie Cartan de Lorraine,  UMR 7502\\
Ile du Saulcy\\ F-57045 Metz, France}
\email{raymond.mortini@univ-lorraine.fr}
 
\author{Amol Sasane}
\address{Department of Mathematics, 
London School of Economics, 
Houghton Street, London WC2A 2AE, U.K.}
\email{sasane@lse.ac.uk}

\begin{abstract}
  It is shown that the algebra $\sH$ of bounded Dirichlet series is
  not a coherent ring, and has infinite Bass stable rank.  As
  corollaries of the latter result, it is derived that $\sH$ has
  infinite topological stable rank and infinite Krull dimension.
\end{abstract}

\maketitle

\section{Introduction}

The aim of this short note is to make explicit two observations about
algebraic properties of the ring $\sH$ of bounded Dirichlet series.
In particular we will show that
\begin{enumerate}
\item $\sH$ is not a coherent ring. (This is essentially an immediate
  consequence of Eric Amar's proof of the noncoherence of the Hardy
  algebra $H^\infty(\mD^n)$ of the polydisk $\mD^n$ for $n\geq 3$ \cite{Ama}.)
\item $\sH$ has infinite Bass stable rank. (This is a straightforward
  adaptation of the first author's proof of the fact that the
  stable rank of the infinite polydisk algebra is infinite \cite{Mor}). As corollaries, we
  obtain that $\sH$ has infinite topological stable rank, and infinite
  Krull dimension.
\end{enumerate}
Before giving the relevant definitions, we briefly mention that $\sH$ is 
a closed Banach subalgebra of the classical Hardy algebra 
$H^{\infty}(\mC_{\scriptscriptstyle >0})$ consisting of all 
bounded and holomorphic functions in the open right half plane 
$$
\mC_{\scriptscriptstyle >0}:= \{s\in \mC:\textrm{Re}(s)>0\},
$$
and it is striking to compare our findings with the corresponding results 
for $H^{\infty}(\mC_{\scriptscriptstyle >0})$:

\begin{center}
\begin{tabular}{|l||l|c|}\hline
  &  $H^\infty(\mC_{\scriptscriptstyle >0})$ & $\sH$\\ \hline\hline
 Coherent? & Yes (See \cite{McVRub}) & No \\ \hline 
 Bass stable rank & $1$ \phantom{aa}(See \cite{Tre}) & $\infty$ \\ \hline
 Topological stable rank & $2$\phantom{aa} (See \cite{Sua}) & $\infty$ \\ \hline 
 Krull dimension & $\infty$\phantom{a} (See \cite{von}) & $\infty$ \\ \hline
\end{tabular}
\end{center}

\medskip 

\noindent Nevertheless the above results for $\sH$ lend support to Harald Bohr's 
idea of interpreting Dirichlet series as functions of infinitely many 
complex variables, a key theme used in the proofs of the main results 
in this note.

We recall the pertinent definitions below.

\subsection{The algebra $\sH$ of bounded Dirichlet series} 

$\sH$ denotes the set of Dirichlet series
\begin{equation}
\label{eq_DS_1}
f(s)=\sum_{n=1}^\infty \frac{a_n}{n^s},
\end{equation}
where $(a_n)_{n\in \mN}$ is a sequence of complex numbers, such that
$f$ is holomorphic and bounded in $\mC_{\scriptscriptstyle >0}$. 
Equipped with pointwise operations and
the supremum norm,
$$
\|f\|_\infty:=\sup_{s\in \mC_{\scriptscriptstyle >0}} |f(s)|, \quad f\in \sH,
$$
$\sH$ is a unital commutative Banach algebra.  In
\cite[Theorem~3.1]{HedLinSei}, it was shown that the Banach algebra
$\sH$ is precisely the multiplier space of the Hilbert space $\calH$
of Dirichlet series 
$$
\displaystyle
f(s)=\sum_{n=1}^\infty\frac{a_n}{n^s} 
$$ 
for which
$$
\displaystyle \|f\|_{\calH}^2:=\sum_{n=1}^\infty |a_n|^2<\infty.
$$
The importance of the Hilbert space $\calH$ stems from the fact that
its kernel function $K_{\calH}(z,w)$ is related to the Riemann zeta
function $\zeta$: 
$$
K_{\calH}(z,w)=\zeta(z+\overline{w}).
$$
For $m\in \mN$, let $\mathscr{H}^\infty_m$ be the closed subalgebra of
$\sH$ consisting of Dirichlet series of the form \eqref{eq_DS_1}
involving only integers $n$ generated by the first $m$ primes
$2,3,\cdots, p_m$.
 
\subsection{$\sH= H^\infty(\mD^\infty)$}

In \cite[Lemma~2.3 and the proof of Theorem~3.1]{HedLinSei}, it was
established that $\sH$ is isometrically (Banach algebra) isomorphic to
$H^\infty(\mD^\infty )$, a certain algebra of functions analytic in
the infinite dimensional polydisk, defined below. As this plays a
central role in what follows, we give an outline of this based on
\cite{HedLinSei}, \cite{Sei} and \cite{MauQue}.

A seminal observation made by H. Bohr \cite{Boh}, is that if we
put
$$
z_1= \frac{1}{2^s}, \;\; 
z_2= \frac{1}{3^s}, \;\; 
z_3= \frac{1}{5^s}, \cdots, 
z_n= \frac{1}{p_n^s}, \cdots, 
$$
where $p_n$ denotes the $n$th prime, then, in view of the Fundamental
Theorem of Arithmetic, formally a Dirichlet series in
$\mathscr{H}^\infty_n$ or $\sH$ can be considered as a power series of
infinitely many variables. Indeed, each $n$ has a unique expansion
$$
n=p_1^{\alpha_1(n)} \cdots p_{r(n)}^{\alpha_{r(n)}(n)},
$$
with nonnegative $\alpha_j(n)$s, and so, from \eqref{eq_DS_1}, we obtain the
formal power series
\begin{equation}
 \label{eq_formal_cap_F}
F(\mathbf{z})=\sum_{n=1}^\infty a_n z_1^{\alpha_1(n)}\cdots z_{r(n)}^{\alpha_{r(n)}(n)},
\end{equation}
where $\mathbf{z}=(z_1,\cdots, z_m)$ or
$\mathbf{z}=(z_1,z_2,z_3,\cdots)$ depending on whether $f$ is a
function in $ \mathscr{H}^\infty_m$ or in $\sH$.  Let us recall  Kronecker's
Theorem on diophantine approximation \cite[Chapter~XXIII]{HarWri}:

\begin{proposition}
For each $m\in \mN$, the map 
$$
t\mapsto (2^{-it}, 3^{-it}, \cdots, p_m^{-it}):(0,\infty)\rightarrow \mT^m
$$ 
has dense range in $\mT^m$, where $\mT:=\{z\in \mC:|z|=1\}$. 
\end{proposition}

\noindent Using the above and the Maximum Principle, it can be shown that for $f\in
\mathscr{H}^\infty_m$,
\begin{equation}
 \label{eq_norm_eq}
 \|f\|_\infty =\|F\|_\infty,
\end{equation}
where the norm on the right hand side is the $H^\infty(\mD^m)$
norm. Here $H^\infty(\mD^m)$ denotes the usual Hardy algebra of
bounded holomorphic functions on the polydisk $\mD^m$, endowed with
the supremum norm:
$$
\|F\|_\infty :=\sup_{\mathbf{z}\in \mD^m}|F(\mathbf{z})|,\quad F\in H^\infty(\mD^m).
$$
In \cite{HedLinSei}, it was shown that this result also holds in the infinite dimensional
case.  In order to describe this
result, we introduce some notation. Let $c_0$ be the Banach space of
complex sequences tending to $0$ at infinity, with the induced norm
from $\ell^\infty$, and let $B$ be the open unit ball of that Banach
space.  Thus with $\mN:=\{1,2,3,\cdots\}$ and $\mD:=\{z\in
\mC:|z|<1\}$,
$$
B=c_0\cap \mD^\mN.
$$
For a point  $\mathbf{z}=(z_1,\cdots, z_m,\cdots)\in B$, we set 
 $
\mathbf{z}^{(m)}:=(z_1,\cdots, z_m,0,\cdots),
$ 
that is, $z_k=0$ for $k>m$. Substituting $\mathbf{z}^{(m)}$ in the
argument of $F$ given formally by \eqref{eq_formal_cap_F},  
we obtain a function
$$
(z_1,\cdots,z_m)\mapsto F(\mathbf{z}^{(m)}),
$$
which we call the $m$th-section $F_m$ (after Bohr's terminology
``$m$te abschnitt''). $F$ is said to be in $H^\infty (\mD^\infty)$ if
the $H^\infty$ norm of these functions $F_m$ are uniformly bounded,
and denote the supremum of these norms to be $\|F\|_\infty$. Using
Schwarz's Lemma for the polydisk, it can be seen that for $m<\ell$,
$$
|F(\mathbf{z}^{(m)})-F(\mathbf{z}^{(\ell)})|\leq 
2 \|f\|_\infty \cdot \max\{|z_j|:m<j\leq \ell\},
$$
and so we may define 
$$
F(\mathbf{z})=\lim_{m\rightarrow \infty} F(\mathbf{z}^{(m)}).
$$
It was shown in \cite{HedLinSei} that \eqref{eq_norm_eq} remains true
in the infinite dimensional case, and so we may associate $\sH$ with
$H^\infty(\mD^\infty)$.

\begin{proposition}[\cite{HedLinSei}]
  There exists a Banach algebra isometric isomorphism $\iota :
  \sH\rightarrow H^\infty(\mD^\infty)$.
\end{proposition}

\subsection{Coherence}

\begin{definition}
Let $R$ be a unital commutative ring, and for $n\in \mN$, let $R^n=R\times \cdots\times R$ ($n$ times). 

\noindent For $\mathbf{f}=(f_1,\cdots, f_n)\in R^n$, a {\em relation} 
$\mathbf{g}$ on $\mathbf{f}$ is an $n$-tuple $\mathbf{g}=(g_1,\cdots, g_n)$ in $R^n$ such that
$$
g_1 f_1+\cdots + g_n f_n=0.
$$
The set of all relations on $\mathbf{f}$ is denoted by  $\mathbf{f}^\perp$. 

\noindent The ring $R$ is said to be {\em coherent} if for each $n$ and each
$\mathbf{f}\in R^n$, the $R$-module $\mathbf{f}^\perp$ is finitely generated.
\end{definition}

A property which is equivalent to coherence is that
the intersection of any two finitely generated ideals in $R$ is
finitely generated, and the annihilator of any element is finitely
generated \cite{Cha}. We refer the reader to the article \cite{Gla} and the monograph
\cite{Gla2} for the relevance of the property of coherence in
commutative algebra. All Noetherian rings are coherent, but
not all coherent rings are Noetherian. (For example, the polynomial
ring $\mC[x_1, x_2, x_3,\cdots ]$ is not Noetherian because the
sequence of ideals $\langle x_1 \rangle \subset \langle x_1,
x_2\rangle \subset \langle x_1, x_2, x_3\rangle \subset \cdots$ is
ascending and not stationary, but $\mC[x_1, x_2, x_3,\cdots ]$ is
coherent \cite[Corollary~2.3.4]{Gla2}.)

In the context of algebras of holomorphic functions in the unit disk
$\mD$, we mention \cite{McVRub}, where it was shown that the
Hardy algebra $H^\infty(\mD)$ is coherent, while the disk algebra
$A(\mD)$ isn't.  For $n\geq 3$, Amar \cite{Ama} showed that the Hardy
algebra $H^\infty (\mD^n)$ is not coherent. (It is worth mentioning that whether
the Hardy algebra $H^\infty(\mD^2)$ of the bidisk is coherent or not
seems to be an open problem.)  Using Amar's result, we will prove the
following result:

\begin{theorem}
\label{main_theorem}
 $\sH$ is not coherent.
\end{theorem}

\subsection{Stable rank}

In algebraic $K$-theory, the notion of (Bass) stable rank of a ring
was introduced in order to facilitate $K$-theoretic computations 
\cite{Bas}.

\begin{definition}
Let $R$ be a commutative ring with an identity element (denoted by $1$). 

\noindent An element $(a_1,\cdots, a_n)\in R^n$ is called {\em unimodular} if
there exist elements $b_1,\cdots, b_n$ in $R$ such that
  $$
 b_1 a_1+\cdots +b_n a_n=1. 
 $$
The set of all unimodular elements of $R^n$ is denoted by $U_n (R)$. 

\noindent We say that $
 a=(a_1,\cdots, a_{n+1})\in U_{n+1}(R)
$ 
is {\em reducible} if there exists an element $(x_1,\cdots, x_n)\in R^n $ such
that
$$
(a_1+x_1 a_{n+1},\;\cdots, \; a_n+ x_n a_{n+1})\in U_n(R).
$$
The {\em Bass stable rank of } $R$ is the least integer $n\in \mN$ for
which every $a\in U_{n+1}(R)$ is reducible. If there is no such
integer $n$, we say that $R$ {\em has infinite stable rank}.
\end{definition}

\noindent Using the same idea as in \cite[Proposition~1]{Mor} (that
the infinite polydisk algebra $A(\mD^\infty)$ has infinite Bass stable
rank), we show the following.

\begin{theorem}
\label{Bass_stable_rank_of_A_DS}
 The Bass stable rank of $\sH$ is infinite.
\end{theorem}

For Banach algebras, an analogue of the Bass stable rank, called 
the topological stable rank, was introduced by Marc Rieffel in \cite{Rie}.

\begin{definition}
  Let $R$ be a commutative complex Banach algebra with unit element
  $1$. The least integer $n$ for which $U_n(R)$ is dense in $R^n$ is
  called the {\em topological stable rank of } $R$. We say $R$ {\em
    has infinite topological stable rank} if no such integer $n$
  exists.
 \end{definition}
 
\begin{corollary}
  The topological stable rank of $\sH$ is infinite. 
\end{corollary}
\begin{proof} This follows from the inequality that the Bass stable
  rank of a commutative unital semisimple complex Banach algebra is at
  most equal to its topological stable rank; see \cite[Corollary
  2.4]{Rie}.
\end{proof}

\begin{definition}
  The Krull dimension of a commutative ring $R$ is the supremum of the
  lengths of chains of distinct proper prime ideals of $R$.
\end{definition}

\begin{corollary}
  The Krull dimension of $\sH$ is infinite. 
\end{corollary}
\begin{proof} This follows from the fact that if a ring has Krull
  dimension $d$, then its Bass stable rank is at most $d+2$; see
  \cite{Hei}.
\end{proof}
 
\section{Noncoherence of $\sH$}

We will use the following fact due to Amar \cite[Proof of
Theorem~1.(ii)]{Ama}.

\begin{proposition}
\label{prop_Amar}
$(z_1-z_2,z_2-z_3)^\perp$ is not a finitely generated
$H^\infty(\mD^3)$-module.
\end{proposition}

\begin{proof}[Proof of Theorem~\ref{main_theorem}] 
The main idea of the proof is that, using the isomorphism $\iota$, 
essentially we boil the problem down to working with
$H^\infty(\mD^\infty)$.  Let
\begin{eqnarray*}
  f_1&:=&\frac{1}{2^s}-\frac{1}{3^s},\\
  f_2&:=& \frac{1}{3^s}-\frac{1}{5^s}.
\end{eqnarray*}
Then $\iota(f_1)=z_1-z_2$ and $\iota(f_2)=z_2-z_3$.  Suppose that
$(f_1,f_2)^\perp$ is a finitely generated $\sH$-module, say by
$$
\left[ \begin{array}{cc} g_1^{(1)}\\g_1^{(2)}\end{array}\right], 
 \cdots ,
\left[ \begin{array}{cc} g_r^{(1)}\\g_r^{(2)}\end{array}\right]\in (\sH)^2.
$$
We will show that the $3$rd section of the image under $\iota$ of the
above elements generate $(z_1-z_2,z_2-z_3)^\perp$ in
$H^\infty(\mD^3)$, contradicting Proposition~\ref{prop_Amar}.  If
$$
\left[\begin{array}{cc} G^{(1)} \\ G^{(2)}\end{array}\right] 
\in (H^\infty(\mD^3))^2 \cap (F_1,F_2)^\perp,
$$
then $F_1 G^{(1)} +F_2G^{(2)}=0$, and by applying $\iota^{-1}$, we see
that
$$
 \left[\begin{array}{cc} \iota^{-1}G^{(1)} \\ \iota^{-1} G^{(2)}\end{array}\right] 
 \in  (f_1,f_2)^\perp.
$$
So there exist $\alpha^{(1)},\cdots,\alpha^{(r)} \in \sH$ such that
$$
\left[\begin{array}{cc} \iota^{-1}G^{(1)} \\ \iota^{-1} G^{(2)}\end{array}\right] 
=
\alpha^{(1)} \left[ \begin{array}{cc} g_1^{(1)}\\g_1^{(2)}\end{array}\right]
+ \cdots +
\alpha^{(r)} \left[ \begin{array}{cc} g_r^{(1)}\\g_r^{(2)}\end{array}\right].
$$
Applying $\iota$, we obtain
$$
\left[\begin{array}{cc} G^{(1)} \\ G^{(2)}\end{array}\right] 
=
\iota(\alpha^{(1)}) 
\left[ \begin{array}{cc} \iota(g_1^{(1)}) \\ \iota(g_1^{(2)})\end{array}\right]
+ \cdots +
\iota(\alpha^{(r)}) 
\left[ \begin{array}{cc} \iota(g_r^{(1)}) \\ \iota(g_r^{(2)}) \end{array}\right].
$$
Finally taking the $3$rd section, we obtain
$$
\left[\begin{array}{cc} G^{(1)}(z_1,z_2,z_3) \\ G^{(2)}(z_1,z_2,z_3)\end{array}\right] 
=
\sum_{j=1}^r (\iota(\alpha^{(j)}))(\mathbf{z}^{(3)}) 
\left[ \begin{array}{cc} (\iota(g_j^{(1)}))(\mathbf{z}^{(3)}) \\ 
(\iota(g_j^{(2)}))(\mathbf{z}^{(3)})\end{array}\right].
$$
So it follows that
$$
\left[ \begin{array}{cc} (\iota(g_1^{(1)}))(\mathbf{z}^{(3)}) \\ 
(\iota(g_1^{(2)}))(\mathbf{z}^{(3)})\end{array}\right],
\cdots ,
\left[ \begin{array}{cc} (\iota(g_r^{(1)}))(\mathbf{z}^{(3)}) \\ 
(\iota(g_r^{(2)}))(\mathbf{z}^{(3)}) \end{array}\right]
$$
generate $(z_1-z_2,z_2-z_3)^\perp$, a contradiction to Amar's result,
Proposition~\ref{prop_Amar}.
\end{proof}

\section{Stable rank of $\sH$}

The proof of Theorem~\ref{Bass_stable_rank_of_A_DS} is a
straightforward adaptation of the first author's proof of the fact that
the Bass stable rank of the infinite polydisk algebra is infinite
\cite[Proposition~1]{Mor}.  In \cite{Mor}, the infinite polydisk algebra
$A(\mD^\infty)$ is the uniform closure of the algebra generated by the
coordinate functions $z_1,z_2,z_3,\cdots$ on the countably infinite
polydisk $\overline{\mD}\times \overline{\mD}\times
\overline{\mD}\times \cdots$.

\begin{proof}[Proof of Theorem~\ref{Bass_stable_rank_of_A_DS}:] 
Fix $n\in \mN$. Let $g\in \sH$ be given by 
\begin{equation}
\label{def_g}
 g(s):= \prod_{j=1}^n \Big(1-\frac{1}{(p_j p_{n+j})^s} \Big) \in \sH.
\end{equation}
Set 
$$
\mathbf{f}:=\Big(\frac{1}{2^s},\cdots,\frac{1}{p_n^s}, g\Big) \in (\sH)^{n+1}.
$$
We will show that $\mathbf{f}\in U_{n+1} (\sH)$ is not
reducible. First let us note that $\mathbf{f}$ is unimodular. Indeed,
by expanding the product on the right hand side of \eqref{def_g}, we
obtain
$$
g=1+\frac{1}{2^s}\cdot  g_1+\cdots +\frac{1}{p_n^s}\cdot g_n,
$$
for some appropriate $g_1,\cdots, g_n\in \sH$. Now suppose that
$\mathbf{f}$ is reducible, and that there exist $h_1,\cdots,h_n\in
\sH$ such that
$$
\Big(\frac{1}{2^s}+gh_1,\cdots,\frac{1}{p_n^s}+gh_n\Big)\in U_n(\sH).
$$
Let $y_1,\cdots , y_n\in \sH$ be such that 
$$
\Big( \frac{1}{2^s}+gh_1\Big)y_1+\cdots+\Big(\frac{1}{p_n^s}+gh_n\Big)y_n=1.
$$
Applying $\iota$, we obtain
\begin{equation}
\label{eq_Bsr_1}
(z_1+\iota(g) \iota(h_1))\iota(y_1)+\cdots+ (z_n+\iota(g) \iota(h_n))\iota(y_n)=1 .
\end{equation}
Let $\mathbf{h}:=(\iota(h_1),\cdots, \iota(h_n))$. For
$\mathbf{z}=(z_1,\cdots, z_n)\in \mC^n$, we define
$$
\mathbf{\Phi}(\mathbf{z})=\left\{ \begin{array}{l} 
 -\mathbf{h}(z_1,\cdots, z_n,\overline{z_1},\cdots, \overline{z_n},0,\cdots) 
\displaystyle\prod_{j=1}^n (1-|z_j|^2)  
\\
\phantom{h(z_1,\cdots, z_n,\overline{z_1},\cdots, \overline{z_n},0,\cdots)}
\textrm{ for }\; |z_j|< 1, \;j=1,\cdots,n,\\
0 \; \textrm{ otherwise}.\phantom{\displaystyle\prod_{j=1}^n}
\end{array}\right.
$$
Then $\mathbf{\Phi}$ is a continuous map from $\mC^n$ into $\mC^n$.
But $\mathbf{\Phi}$ vanishes outside $ \mD^n$, and so
$$
 \max_{\mathbf{z}\in \mD^n }\|\mathbf{\Phi}(\mathbf{z})\|_2
= \sup_{\mathbf{z}\in \mC^n} \|\mathbf{\Phi}(\mathbf{z})\|_2.
$$
This implies that there must exist an $r\geq 1$ such that $\mathbf{\Phi}$
maps $K:=r \overline{\mD}^n$ into $K$. As $K$ is compact and convex,
by Brouwer's Fixed Point Theorem it follows that there exists a
$\mathbf{z}_*\in K$ such that
$$
\mathbf{\Phi}(\mathbf{z}_*)=\mathbf{z}_*.
$$
Since $\mathbf{\Phi}$ is zero outside $\mD^n$, we see that
$\mathbf{z}_*\in \mD^n$.  Let $\mathbf{z}_*=(\zeta_1,\cdots,
\zeta_n)$. Then for each $j\in \{1,\cdots, n\}$, we obtain
\begin{eqnarray}
\nonumber
0&=&\zeta_j + (\iota(h_j))(\zeta_1,\cdots, \zeta_n, 
\overline{\zeta_1},\cdots, \overline{\zeta_n},0,\cdots) 
\prod_{k=1}^n (1-|\zeta_k|^2)
\\
\label{eq_Bsr_2}  
&=& \zeta_j + (\iota(h_j) \iota(g))(\zeta_1,\cdots, \zeta_n, 
\overline{\zeta_1},\cdots, \overline{\zeta_n},0,\cdots).
\end{eqnarray}
But from \eqref{eq_Bsr_1}, we know that  
$$
\sum_{j=1}^n (z_j + \iota(h_j) \iota(g)) \iota(y_j)= 1 ,
$$
and this contradicts \eqref{eq_Bsr_2}.  As the choice of $n\in \mN$
was arbitrary, it follows that the Bass stable rank of $\sH$ is
infinite.
\end{proof}

\noindent {\bf Acknowledgement}: Useful discussions with Anders Olofsson (Lund
University) are gratefully acknowledged by the second author.

\end{document}